\documentclass[12pt]{article}
\usepackage[dvips]{graphicx}
\usepackage{amsmath,amsfonts,amssymb,amsthm,color}

\usepackage{natbib}

\newtheorem{theorem}{Theorem}[section]

\newtheorem{proposition}{Proposition}[section]

\newtheorem{definition}{Definition}[section]
\newtheorem{example}{Example}[section]

\newtheorem{condition}{Condition}[section]

\usepackage{bm}
\bmdefine{\Bt}{t}
\bmdefine{\BX}{X}
\bmdefine{\BY}{Y}
\bmdefine{\BZ}{Z}
\bmdefine{\BB}{B}
\bmdefine{\BM}{M}
\bmdefine{\BD}{D}
\bmdefine{\Bi}{i}
\bmdefine{\Bj}{j}
\bmdefine{\Bx}{x}
\bmdefine{\By}{y}
\bmdefine{\Bz}{z}
\bmdefine{\Bv}{v}
\bmdefine{\Bw}{w}
\bmdefine{\Bn}{n}
\bmdefine{\Ba}{a}
\bmdefine{\Bb}{b}
\bmdefine{\Bc}{c}
\bmdefine{\Be}{e}
\bmdefine{\Bu}{u}
\bmdefine{\Bp}{p}
\bmdefine{\Bzero}{0}
\bmdefine{\Bone}{1}

\newcommand{\cB}{{\mathcal B}}
\newcommand{\cF}{{\mathcal F}}
\newcommand{\cI}{{\mathcal I}}

\newcommand{\supp}{\mathop{\mathrm{supp}}}

\numberwithin{equation}{section}

\title{Connecting tables with zero-one entries by a subset of a Markov basis}
\author{
Hisayuki Hara\footnote{
Department of Technology Management for Innovation, 
University of Tokyo} \ and
Akimichi Takemura\footnote{
Graduate School of Information Science and Technology, 
University of Tokyo}
\footnote{JST, CREST} }

\date{November 2009}

\begin{document}
\maketitle

\begin{abstract}
We discuss connecting tables with zero-one entries 
by a subset of a Markov basis.  In this paper, as a Markov basis 
we consider the Graver basis, which corresponds to the unique minimal Markov 
basis for the Lawrence lifting of the original configuration. 
%specifying which marginals are fixed.
%defining the toric ideal.
Since the Graver basis tends to be large, it is of interest to
clarify conditions such that a subset of the Graver basis, in particular
a minimal Markov basis itself, connects tables with zero-one entries.
We give some theoretical results on the connectivity of tables 
with zero-one entries. We also study some common models, where a minimal 
Markov basis for tables without the zero-one restriction does 
not connect tables with zero-one entries.
\end{abstract}

\noindent
{\it Key words:}  Graver basis, Latin squares, logistic
regression, Rasch model 

\section{Introduction}
\label{sec:intro}
Markov bases methodology initiated by \cite{diaconis-sturmfels} for
performing conditional tests of discrete exponential family models
have been extensively studied in recent years.  
The set of contingency tables sharing values of the sufficient
statistic is called a fiber. 
A Markov basis guarantees connectivity of all fibers by definition. 
Since the size of a Markov basis tends to be large for large-scale
problems, 
researchers are interested in a subset of Markov basis %for a specific
%fiber.  
to connect specific fibers. 
In most applications of Markov basis there are no restrictions
on the cell counts.  However in some problems, the counts are either
zero or one.  The most well-known case is the Rasch model
(\cite{Rasch1980}) used in educational statistics.  

The Rasch model can be interpreted as a logistic regression (logit
model),  where the number
of trials is just one for each combination of covariates.  In this
model, tables with zero-one entries (zero-one tables) 
are elements of a specific fiber, where the
marginal frequencies corresponding to the response variable are all equal
to one in the logistic regression.  

In other cases, zero-one tables appear as truncation or dichotomization
of a variable, where for example only an occurrence or non-occurrence of
certain large event is recorded. 
A convenient statistical model for zero-one tables is a log-linear model
for contingency tables, where the support of the distribution is restricted
to zero-one tables.  Then we can use the Markov basis methodology for
conditional tests of the fit of the model.

Two-way zero-one tables with structural zeros arise in many practical
problems in ecological studies and social networks 
and exact tests of quasi-independence models via a Markov basis has been 
studied for some specific problems 
(e.g. \cite{rao-etal-1996sankhya,roberts-2000SocialNetworks}). 

Another source of zero-one tables is the set of
incidence matrices satisfying certain combinatorial restrictions.
For example, the set of Latin squares and Sudoku tables can be
considered as a set of zero-one tables with fixed marginals.
From combinatorial viewpoint it is of interest to construct a
connected Markov chain over the set of these tables.

Note that a minimal Markov basis without the zero-one restriction may
not connect zero-one tables, because by applying a
move from the Markov basis, some cells may contain frequencies greater
than one.  On the other hand, as clarified in Proposition
\ref{prop:graver} in Section \ref{sec:main}, the set of square-free
moves of the Graver basis   
%(see Section \ref{subsec:notation}) 
connects tables with zero-one entries.  Therefore
it is of interest to study when a minimal Markov basis connects zero-one 
tables, and if this is not the case, to find a subset of the Graver
basis connecting zero-one tables.  

In this paper we give some theoretical results on the connectivity of tables 
with zero-one entries.  Unfortunately we found that our sufficient
conditions for connectivity are satisfied only in a few examples.
Therefore we investigate some common models, where a minimal 
Markov basis for tables without the zero-one restriction does 
not connect tables with zero-one entries.

The organization of the paper is as follows.
For the rest of this section we summarize our notation and preliminary facts.
In  Section \ref{sec:main}  we give some theoretical results on 
connectivity of zero-one tables with a minimal Markov basis and
with some other subsets of the Graver basis.
In Section \ref{sec:models} we study connectivity of zero-one tables in 
some common models for contingency tables, including the Rasch model, 
its multivariate version and the quasi-independence model. 
%and the logistic regression.    
We also discuss Latin squares.
We conclude the paper with some remarks
in Section \ref{sec:remarks}.

%This is the particular case of 
%This is essentially the problem of Graver basis.

\subsection{Notation and preliminary facts}
\label{subsec:notation}

Here we set up our notation and summarize preliminary facts on Markov
and Graver bases.
We mostly follow the notation in  \cite{decomposable}.
Let ${\mathcal I}$ denote the set of cells of a table, 
where $i \in \cI$ is usually a multi-index.
Let $\Delta$ be the set of variables.
Then a cell $i$ is considered as a $\vert \Delta \vert$ dimensional
vector $i := (i_d)_{d \in \Delta}$. 
Denote by $I=\vert {\mathcal I} \vert$ the number of cells. 
For a subset $D \subset \Delta$, denote by 
$i_D := (i_d)_{d \in D}$ and ${\mathcal I}_D$ a $D$-marginal cell and
the set of $D$-marginal cells, respectively.  
Define ${\mathcal I}_D := \prod_{d \in D} {\mathcal I}_d$ and 
$I_D := \prod_{d \in D} I_d$. 

A contingency table or a frequency vector is denoted by 
$\Bx=(x(i))_{\{ i\in {\mathcal I}\}}$.
Let $x(i_D)$ denote a marginal frequency, i.e. 
$x(i_D) = \sum_{i_{D^c} \in {\mathcal I}_{D^c}} x(i_D, i_{D^c})$.
For a given $\Bx$, $\supp(\Bx)=\{ i \mid x(i)>0 \}$ denotes the set of
positive cells of $\Bx$. 
Given a loglinear model (more precisely a toric model), the
sufficient statistic $\Bt$ can be written as $\Bt = A\Bx$ for some integral matrix $A$.
We call $A$ a configuration of the model.  
$I_A$ denotes the toric ideal of $A$. 
Assume that $I_A$ is homogeneous, i.e. 
there exists a vector $\bm{w}$ such that 
\begin{equation}
 \label{homo-toric}
  \bm{w}^\prime A = (1,\ldots,1)
\end{equation}
(Lemma 4.14 in \cite{sturmfels1996}). 
The set $\cF_\Bt= \{ \Bx \ge 0 \mid \Bt=A \Bx\}$ of contingency tables 
with the common sufficient statistic $\Bt$ is called a {\it fiber}.

An integer vector $\Bz$ is called a move if $A\Bz=\Bzero$.
$|\Bz|=\sum_{i\in \cI} |z(i)|$ denotes the $L_1$-norm of $\Bz$.
Separating  positive elements and negative elements of $\Bz$, we write
$\Bz=\Bz^+ - \Bz^-$, where $\Bz^+$ is the positive part of $\Bz$ and
$\Bz^-$ is the negative part of $\Bz$.  
%The sample size of $\Bz^+$ ($\Bz^-$) is called degree of $\Bz$.
The total sum of frequencies in $\Bz^+$ ($\Bz^-$) is called degree of
$\Bz$.  
%We denote a square-free move $\Bz$ of degree $m$ by  
%$\Bz=(i_1)(i_2)\cdots(i_m) - (j_1)(j_2)\cdots(j_m)$, 
%where 
%$\supp(\Bz^+) = \{i_1,\ldots,i_m\}$ and 
%$\supp(\Bz^-) = \{j_1,\ldots,j_m\}$.
For two moves $\Bz_1, \Bz_2$, the sum $\Bz_1 + \Bz_2$ is called {\it conformal}
if there is no cancellation of signs in $\Bz_1 + \Bz_2$, i.e., 
$\emptyset = \supp(\Bz_1^+)\cap \supp(\Bz_2^-)= \supp(\Bz_1^-)\cap \supp(\Bz_2^+)$.
A move $\Bz$ which can not be written as a conformal sum of two (non-zero) moves
is called {\it primitive}.  The set of primitive moves is finite and it is called the
Graver basis of $I_A$.  We denote the Graver basis as $\cB_{{\rm GR}}$.

Let $E_I$ denote the $I \times I$ identity matrix. 
The configuration 
\begin{equation}
\label{eq:lawrence}
\Lambda(A) = 
\begin{pmatrix} A & 0 \\
               E_I & E_I
\end{pmatrix}
\end{equation}
is called the Lawrence lifting of $A$.  In statistical terms, the Lawrence lifting
corresponds to the logistic regression, where the interaction effects of the covariates
are specified by $A$.
It is known
(\cite[Theorem 7.1]{sturmfels1996}) that the unique minimal Markov
basis of $I_{\Lambda(A)}$ coincides with the Graver basis of $I_A$.
 
A finite set of moves $\cB$ is {\it distance reducing} (\cite{takemura-aoki-2005bernoulli})
if for all $\Bt$ and for all $\Bx,\By\in \cF_\Bt$
there exists an element $\Bz \in \cB $ and $\epsilon =\pm 1$ such that
% such that $\epsilon \Bz$ is applicable 
% to $\Bx$ or $\By$ and we can decrease the distance, i.e., 
\[
\Bx + \epsilon \Bz\in {\mathcal F}_{\Bt},\  |\Bx + \epsilon \Bz - \By| < |\Bx- \By| \quad \mbox{or} \quad
\By + \epsilon \Bz\in {\mathcal F}_{\Bt},\  |\Bx - (\By+\epsilon \Bz) | < |\Bx -\By|.
\]
If $\cB$ is distance reducing, it is obviously a Markov basis and 
we call $\cB$  a distance reducing Markov basis.
Furthermore ${\mathcal B}$ is {\it strongly distance reducing} 
if for all $\Bt$ and for all $\Bx,\By\in \cF_\Bt$
there exist elements $\Bz_1, \Bz_2 \in
{\mathcal B}$ and $\epsilon_1, \epsilon_2  =\pm 1$ such that 
$\Bx+ \epsilon_1 \Bz_1, \By+ \epsilon_2 \Bz_2 \in {\mathcal F}_{\Bt}$, 
$|\Bx + \epsilon_1 \Bz_1- \By| < |\Bx-\By|$
and
$|\Bx -(\By+\epsilon_2 \Bz_2 )| < |\Bx-\By|$.

Since we are considering zero-one tables in this paper, let us denote
\begin{equation}
\label{eq:zero-one-for-A}
\tilde  \cF_{\Bt}= \{ \Bx \mid \Bt=A \Bx,\ x(i)=0 \; \text{or}\; 1 \}.
\end{equation}
As in the usual setting for Markov bases, we call a finite set $\cB$
of moves a Markov basis for zero-one tables, if $\cB$ connects all
fibers $\tilde  \cF_\Bt$.  If $\cB$ is distance reducing for
zero-one tables, then it is a distance reducing Markov basis for zero-one tables.
Since there are $2^{|\cI|}$ zero-one tables,
there are only finitely many fibers and finitely many 
differences of two elements belonging to the same fiber.  Therefore
the set of these differences is the largest trivial Markov basis. However
this set is clearly too large and we are interested in a much smaller
set of moves connecting all fibers $\tilde  \cF_\Bt$.

\section{Some theoretical results}
\label{sec:main}

%We assume homogeneity.

The starting point of our investigation of 
connectivity of zero-one tables is the following basic fact on
the Graver basis ${\mathcal
  B}_{\rm GR}$ for $I_A$.
%The basic result for zero-one tables is the following fact.

\begin{proposition}
 \label{prop:graver}
 Let $\cB_0$ denote the set of square-free moves of the Graver
 basis ${\mathcal   B}_{\rm GR}$ of $I_A$. Then 
 $\cB_0$ is 
 strongly distance reducing
 for tables with zero-one entries.
% (hence a Markov basis) for tables of zeros
% and ones, i.e., for every fiber and every two zero-one tables $\Bx, \By$ in the
% same fiber, there exists $\Bz$ and $\Bz'$ in $\cB_0$, such
% that $\Bx+\Bz$ is a zero-one table with $|\Bx+\Bz-\By|<|\Bx-\By|$
% and
% $\By+\Bz'$ is a zero-one table with $|\Bx- (\By+\Bz')|<|\Bx-\By|$.
\end{proposition}

\begin{proof}
Let $\Bx$, $\By$ be two zero-one tables of the same fiber.
They  are connected by a conformal sum of primitive moves 
%oves of the Graver basis:
\begin{equation}
\label{eq:conformal-sum}
\By=\Bx + \Bz_1 + \dots + \Bz_K.
\end{equation}
Since there is no cancellation of signs on the right-hand side,
once an entry greater than or equal to 2 appears in an intermediate sum of
the right-hand side, it can not be canceled.  Therefore it follows 
that  
$\Bx + \Bz_1 + \dots + \Bz_k \in \tilde{F}_\Bt$ for $k = 1,\ldots,K$ and 
$\Bz_1,\dots, \Bz_K \in \cB_0$.  
Since there are no sign cancellations in (\ref{eq:conformal-sum}), 
$\Bz_1,\dots, \Bz_K$ can be added to $\Bx$ in any order and
$-\Bz_1,\dots, -\Bz_K$ can be added  to $\By$ in any order.
Therefore $\cB_0$ is strongly distance reducing.
% Note that each of
% $\Bz_1,\dots, \Bz_K$ can be used as $\Bz$ and 
% each of $-\Bz_1,\dots, -\Bz_K$ can be used as $\Bz'$.
\end{proof}

Since the Graver basis tends to be large, we are interested
in conditions for connecting tables with zero-one entries with
a subset of the Graver basis.
We consider the following condition.

\begin{condition}[Existence of strong crossing pattern] 
 \label{cond:1} 
Let ${\Be_i}$ denote the frequency vector with just $1$ frequency in the
$i$-th cell and $0$ otherwise.  
For every fiber and every $\Bx, \By$, $\Bx\neq \By$, in the same fiber, 
there exist distinct cells $i_1, i_2, i_3, i_4$ such that
$x(i_1)> y(i_1), x(i_2)> y(i_2), x(i_3) < y(i_3), x(i_4) \le y(i_4)$
or 
$y(i_1)> x(i_1), y(i_2)> x(i_2), y(i_3) < x(i_3), y(i_4) \le x(i_4)$
and 
\begin{equation}
\label{eq:goodmoves}
\Bz= \Be_{i_3} + \Be_{i_4}-\Be_{i_1}-\Be_{i_2}
\end{equation}
is a move.
\end{condition}

Note the set $\cB$ 
of the moves $\Bz$ in (\ref{eq:goodmoves})
forms a distance reducing Markov basis.  Therefore Condition 
\ref{cond:1} is a
sufficient condition for existence of a distance reducing Markov basis
consisting of square-free  moves of degree two.  
% We have used
% Condition 
% \ref{cond:1} in our previous studies  (e.g.\ \cite{aoki-takemura-2005jscs},
% \cite{decomposable})
% for showing the existence of distance reducing Markov basis
% consisting of square-free of degree two.
However existence of such a Markov basis 
does not imply Condition \ref{cond:1}.  
We discuss this point at the end of this section.
%Appendix.
%
%We now state the main theorem of this paper.
Under Condition \ref{cond:1} we have the following result.

\begin{theorem}
\label{thm:1}
Under Condition \ref{cond:1}, the set $\cB$ of moves
(\ref{eq:goodmoves})
is distance reducing for tables with zero-one entries.
\end{theorem}

\begin{proof}
Let $\Bx$, $\By$ be two zero-one tables in the same fiber.
%We first ignore the restriction
%that the entries of the tables are  restricted to $\{0,1\}$.  
%% Each element $\Bz$ of $\cB$ 
By Condition \ref{cond:1}, we can find distinct cells $i_1$, $i_2$,
 $i_3$, $i_4$ such that 
\begin{align*}
& x(i_1) \ge y(i_1)+1, \ x(i_2)\ge y(i_2)+1  \\
& \qquad \Rightarrow \quad
x(i_1) = 1, \ x(i_2)= 1, \ y(i_1)=0, y(i_2)=0
\end{align*}
and 
$0 \le x(i_3)< y(i_3)\le 1 \Rightarrow x(i_3)=0, y(i_3)=1$.  
If $x(i_4)=0$ then we can add
$\Bz= \Be_{i_3} + \Be_{i_4}-\Be_{i_1}-\Be_{i_2}$ to $\Bx$
and reduce the $L_1$-distance by four.  Furthermore 
$\Bx+\Bz$ is a table of zeros and ones.

It remains to consider the case $x(i_4)=1$.  Since 
$x(i_4) \le y(i_4)$, we have $y(i_4)=1$.  Therefore
$y(i_1)=0, y(i_2)=0, y(i_3)=1, y(i_4)=1$.  Then we can subtract
$\Bz$ from $\By$ and $\By-\Bz$ is a table of zeros and ones.
Furthermore
$|\Bx-(\By-\Bz)|=|\Bx-\By|-2$.  Therefore under Condition \ref{cond:1} we can
reduce the distance always by at least $2$.  Therefore $\cB$ is 
distance reducing for fibers of zero-one tables.
\end{proof}

Theorem \ref{thm:1} is simple and effective to prove that a
particular Markov basis connects zero-one tables for some simple
configurations.  We now present several generalizations of
Theorem \ref{thm:1}.
% for theoretical interest as well as possible uses
% for some complicated configurations.
The following proposition is an obvious extension of Theorem
\ref{thm:1} and we omit a proof.

\begin{proposition}
 \label{prop:gen1}
 Assume that there exists a positive integer $M$, such that
 for every fiber and every $\Bx, \By$, $\Bx\neq \By$, in the same fiber 
 there exists a positive integer $m\le M$ and 
 distinct cells $i_1,\dots, i_{2m}$ such  that
 \begin{equation}
  \label{eq:goodm}
   \Bz= \sum_{j=m+1}^{2m} \Be_{i_j} - \sum_{j=1}^m \Be_{i_j}
 \end{equation}
 is a move
 such that at least one of the following conditions hold:
 i)
 $x(i_j)> y(i_j)$, $j=1,\dots, m$, $x(i_j) < y(i_j)$,
 $j=m+1,\dots, 2m-1$,  $x(i_{2m})\le  y(i_{2m})$, or
 ii)
 $y(i_j)> x(i_j)$, $j=1,\dots, m$, $y(i_j) < x(i_j)$,
 $j=m+1,\dots, 2m-1$,  $y(i_{2m})\le  x(i_{2m})$.
 Then the set $\cB$ of moves $\Bz$ in 
 (\ref{eq:goodm})
 is distance reducing for tables with zero-one entries.
\end{proposition}

Proposition \ref{prop:gen1} suggests a possibility to choose a
subset of $\cB_0$ of Proposition \ref{prop:graver}, which still
guarantees the connectivity of tables with zero-one entries.
Let $\cB$ be a subset of $\cB_0$ with the following property.

\begin{condition}[Generalized strong crossing pattern for the Graver basis]
\label{cond:1a}
For every element $\Bz=\Bz^+ - \Bz^- \in \cB_0\setminus \cB$, 
there exists a move 
$\Bz'= \sum_{j=m+1}^{2m} \Be_{i_j} - \sum_{j=1}^m \Be_{i_j} \in \cB$
such that  $i_1,\dots, i_{2m}$ are distinct 
and at least one of the
following conditions hold:
i)
$\Bz^+(i_j)> \Bz^-(i_j)$, $j=1,\dots, m$, $\Bz^+(i_j) < \Bz^-(i_j)$,
$j=m+1,\dots, 2m-1$,  $\Bz^+(i_{2m})\le  \Bz^-(i_{2m})$, or
ii) 
$\Bz^-(i_j)> \Bz^+(i_j)$, $j=1,\dots, m$, $\Bz^-(i_j) < \Bz^+(i_j)$,
$j=m+1,\dots, 2m-1$,  $\Bz^-(i_{2m})\le  \Bz^+(i_{2m})$.
\end{condition}

Combining Proposition \ref{prop:graver} and Proposition 
\ref{prop:gen1} we have the following proposition.

\begin{proposition}
\label{prop:gen2}
If $\cB$ satisfies Condition \ref{cond:1a}, then $\cB$ is
distance reducing for tables with zero-one entries.
\end{proposition}

\begin{proof}
As in the proof of Proposition \ref{prop:graver}, consider 
(\ref{eq:conformal-sum}), where
$\Bx$, $\By$ are two zero-one tables in the same fiber.
By induction on the number $K$ of primitive moves,
%of elements $K$ of the Graver basis, 
it suffices to prove the distance reduction 
for $\By = \Bx+\Bz_1$.
% in replacing the operation $$ with 
% moves from $\cB$.
 By the same argument as in Lemma 2.4 of \cite{takemura-aoki-2004aism},
 it suffices  
 to check the distance reduction by moves from $\cB$ in 
 moving from $\Bz_1^-$ to $\Bz_1^+$.
 If $\Bz_1\in \cB$, we can reduce the distance at once.  If 
 $\Bz_1 \in \cB_0 \setminus \cB$, 
we can find $\Bz \in \cB$ which can be applied
either to $\Bz_1^-$ or $\Bz_1^+$ such that  $|\Bz_1|$ is reduced.  The resulting move
can now be decomposed into a conformal sum of primitive moves and we can recursively
use the distance reduction argument.  This proves the proposition.
\end{proof}

By Proposition \ref{prop:gen2}, once $\cB_0$ is given  we can
remove some elements from $\cB_0$ and obtain a smaller set of moves $\cB$ as follows.
Find a pair $\Bz, \tilde \Bz\in \cB_0$, 
$\Bz \neq \tilde \Bz$, such that $\Bz + \tilde \Bz$ has just one sign
cancellation, i.e.\  
there is only one cell $i$ such that 
%$z(i) \tilde z(i) < 0$.  
$z(i) \tilde z(i) = -1$.  
If $\Bz+\tilde \Bz\in \cB_0$ 
then we can remove $\Bz+\tilde \Bz$ from $\cB_0$ 
and still guarantee connectivity of zero-one tables.

% \begin{condition}  (Existence of strong  crossing pattern for the Graver basis)\\
% \label{cond:1a}
% \end{condition}
%\section{Distance reducing and crossing pattern}
% As discussed in Section \ref{sec:main} 
% existence of distance reducing Markov basis consisting of square-free
% moves of degree two does not imply Condition \ref{cond:1}.  

As the last topic of this section 
we clarify the interpretation of 
Condition \ref{cond:1}
by discussing a weaker condition
which is equivalent to the existence of distance reducing Markov basis
consisting of square-free moves of degree two.
In our previous works  (e.g.\ \cite{aoki-takemura-2005jscs},
\cite{decomposable}) we have obtained such Markov bases 
and used a similar argument as in the proof of Theorem \ref{thm:1}.
By omitting the requirement $x(i_4) \le y(i_4)$ in Condition \ref{cond:1}
consider the following weaker condition:

\begin{condition}[Existence of weak crossing pattern]
\label{cond:2}
For every fiber and every $\Bx, \By$, $\Bx\neq \By$, in the same fiber,
there exist distinct cells $i_1, i_2, i_3,i_4$ such that
$x(i_1)> y(i_1), x(i_2)> y(i_2), x(i_3) < y(i_3)$
or 
$y(i_1)> x(i_1), y(i_2)> x(i_2), y(i_3) < x(i_3)$
and 
$\Bz= \Be_{i_3} + \Be_{i_4}-\Be_{i_1}-\Be_{i_2}$ is a move.
\end{condition}

% Under Condition \ref{cond:2} the set of moves $\cB$  is still a distance
% reducing Markov basis.  
We now show that Condition \ref{cond:2} is equivalent to
the existence of a distance reducing Markov basis consisting of square-free
moves of degree two.

\begin{proposition}
There exists a distance reducing Markov basis 
consisting of square-free  moves of degree two if and only if
Condition \ref{cond:2} holds.
\end{proposition}

\begin{proof}
It is clear that  under Condition \ref{cond:2} the set $\cB$ of moves in
(\ref{eq:goodmoves}) is distance reducing.  Therefore it suffices to show the converse.
Let $\cB$ be a distance reducing
Markov basis consisting of square-free  moves of degree two.
Let $\Bx, \By$, $\Bx\neq \By$, be in the same fiber.  We can find
$\pm \Bz\in \cB$ such that $\Bz$ is applicable to $\Bx$ or $\By$ and
$|(\Bx+\Bz)-\By| < |\Bx - \By|$ or 
$|\Bx-(\By+\Bz)| < |\Bx - \By|$, respectively.
For  $\Bx, \By$ in the same fiber and for 
distinct indices $i_1, i_2, i_3, i_4$ let 
\begin{align*}
g(i_1, i_2, i_3, i_4, \Bx, \By)&=I(x(i_1)> y(i_1))
+ I(x(i_2)> y(i_2))\\
& \qquad
+ I(x(i_3) < y(i_3)) + I(x(i_4) < y(i_4))-2,
\end{align*}
where $I(E)$ denotes the indicator function of the event $E$.
When $\Bz$ can be added to $\Bx$, we have
\[
|\Bx - \By| - |(\Bx+\Bz) - \By|=  2 g(i_1, i_2, i_3, i_4, \Bx, \By).
\]
Therefore 
\[
|\Bx - \By| - |(\Bx+\Bz) - \By| > 0 \quad \Leftrightarrow \quad
g(i_1, i_2, i_3, i_4, \Bx, \By)>0,
\]
i.e., at least 3 inequalities among $x(i_1)> y(i_1),
x(i_2)> y(i_2), x(i_3) < y(i_3), x(i_4) < y(i_4)$ hold.
It is easy to see that then Condition \ref{cond:2} holds.
Similarly if $\Bz$ can be added to $\By$ and
$|\Bx-(\By+\Bz)| < |\Bx - \By|$, then Condition \ref{cond:2} holds.
\end{proof}

\section{Connectivity results for some models}
\label{sec:models}

In this section we investigate connectivity of zero-one tables for some common models
for contingency tables.

\subsection{Rasch model}
\label{subsec:Rasch}
Rasch model (\cite{Rasch1980}) has long received much attention in the
item response theory.  
Suppose that $I$ persons take a test with $J$ dichotomous questions.
Let $x_{ij}\in \{0,1\}$ be a response to the $j$th question of the $i$th
person.  
Hence the $I \times J$ table $\bm{x} = (x_{ij})$ is
considered as a two-way contingency table with zero-one entries.
Assume that each $x_{ij}$ is independent.
Then the Rasch model is expressed as
\begin{equation}
 \label{model:Rasch_1dim}
%  \mathrm{Pr}(x_{ij}=1) = 
  P(x_{ij}=1) = 
  \frac{\exp (\alpha_i - \beta_j)}{1+\exp (\alpha_i - \beta_j)}, 
\end{equation}
where $\alpha_i$ is an individual's latent ability parameter and 
$\beta_j$ is an item's difficulty parameter.
%Suppose that each $x_{ij}$ is independent.
Then the set of row sums $x_{i+} = \sum_{j=1}^J x_{ij}$ 
and column sums $x_{+j} = \sum_{i=1}^I x_{ij}$ 
is the sufficient statistic for $\alpha_i$ and $\beta_j$.

The Rasch model has been extensively studied and practically used for
evaluating educational and psychological tests. 
Many inference procedures have been developed
(e.g. \cite{Glas-Verhelst1995}) and most of them rely on asymptotic
theory.  
However, as \cite{Rasch1980} pointed out, 
a sufficiently large sample size is not necessarily expected in practice.
In such cases the asymptotic inference may be inappropriate. 

\cite{Rasch1980} proposed to use an exact test procedure.
As mentioned in \cite{Rasch1980}, 
the conditional distribution of zero-one tables given person scores
and item totals is easily shown to be uniform. 
In order to implement exact test for Rasch model via Markov basis
technique, we need a set of moves which connects every fiber of
two-way zero-one tables with fixed row and column sums. 
\cite{Ryser1957} first showed that the set of basic moves 
\[
 \begin{array}{rrr}
  & i & i^\prime \\
  j & 1 & -1 \\
  j^\prime & -1 & 1\\
 \end{array}
\]
in two-way complete independence model connects any fiber of zero-one
tables with fixed row and column sums. 
Since then, many Monte Carlo procedures via Markov basis technique to
compute distribution of test statistics %under the null hypothesis of 
%the Rasch model have been proposed 
to test the goodness-of-fit of the Rasch model have been proposed 
(e.g.\ \cite{Besag-Clifford1989}, 
\cite{ponocny2001}, \cite{Cobb-Chen2003}). 
\cite{chen-small} provided a computationally more efficient Monte
Carlo procedure for implementing exact tests by using sequential
importance sampling.  

In the framework of the present paper, the
connectivity by basic moves is a consequence of Theorem \ref{thm:1}
and Proposition \ref{prop:graver}.
The Rasch model can be regarded as the Lawrence lifting of the independence model
for $I\times J$ tables.  
Assume that $i_1,i_2,\ldots,i_r$ and $j_1,j_2,\ldots,j_r$
are all distinct. 
Denote $i_{[r]}=(i_1,\ldots,i_r)$, $j_{[r]}=(j_1,\ldots,j_r)$. 
Then a loop of degree $r$ 
 \[
 \bm{z}_r(i_{[r]};j_{[r]}) 
 = \{z_{ij}\}, \quad
 1 \leq i_1,\ldots ,i_r \leq I, \ 1 \leq j_1,\ldots ,j_r \leq J,
 \] 
 is defined by a move such that 
 \[
 \begin{array}{l}
  z_{i_1j_1} = z_{i_2j_2} = \cdots = z_{i_{r-1}j_{r-1}} = z_{i_rj_r} = 1,\\
  z_{i_1j_2} = z_{i_2j_3} = \cdots = z_{i_{r-1}j_r} = z_{i_rj_1} = -1,
 \end{array}
\]
and all the other elements are zero (e.g. \cite{aoki-takemura-2005jscs}). 
A loop of degree $r$ is written as 
\[
 \bm{z} =
 \begin{array}{c|rrrrrr|}
  \multicolumn{1}{r}{}
  & \multicolumn{1}{r}{j_1} & j_2 & \cdots & \cdots & j_{r-1} 
   & \multicolumn{1}{r}{j_r}\\
\cline{2-7}
  i_1 & 1 & -1 & 0  & 0 & \dots &0 \\
  i_2 & 0 & 1  & -1 & 0 & \dots &0\\
  \vdots & \multicolumn{1}{|c}{\vdots} & & & & & \multicolumn{1}{c|}{\vdots}\\
%  0 & 0  & 1 & -1 & \dots &0\\
  i_{r-1} & 0 & 0  & \dots  & \dots & 1 & -1 \\
  i_{r} & -1 & 0 & \dots & \dots  &  0 & 1 \\ \cline{2-7}
 \end{array}\; . 
\]

From p.382 of \cite{diaconis-sturmfels} we know that the set of loops of
degree $r$,  
%$2n$, 
$r \le \min(I,J)$
%, such as
forms the Graver basis for the complete independence model of  
$I \times J$  contingency tables. 
Since the set of basic moves satisfies Condition 
\ref{cond:1}, where $\Bx$ is the positive part and $\By$ is the negative part of
these loops, it follows that the set of basic moves connects $I\times J$ 
zero-one tables with fixed row and column sums.

\subsection{Many-facet  Rasch model}
%\subsection{Multivariate Rasch models}
% 3x3x3 case,  Graver basis
%[to be filled by Hara]
The many-facet Rasch model is an extension of the Rasch model to
multiple items and polytomous responses (e.g.\ \cite{linacre1989},
\cite{linacre1994}) and has also been extensively used in practice for
evaluating essay exams and scoring systems of judged sports
(e.g. \cite{yamaguchi1999}, \cite{zhu-ennis-chen-1998},
\cite{basturk-2008}).  

Suppose that $I_1$ articles are rated by $I_2$ reviewers from $I_3$
aspects on the grade of $I_4$ scales from $0$ to $I_4-1$.
%Let $i = (i_1,\ldots,i_4)$.
$x_{i_1 i_2 i_3 i_4}=1$ if the reviewer $i_2$ rates the article  $i_1$
as the $i_4$th grade from the aspect  $i_3$  and otherwise 
$x_{i_1 i_2 i_3 i_4}=0$. 
Then $\bm{x}=\{x_{i_1 i_2 i_3 i_4}\}$ is an 
$I_1 \times I_2 \times I_3 \times I_4$ zero-one table.
We note that 
$\bm{x}$ satisfies 
$x_{i_1 i_2 i_3 +}:=\sum_{i_4=0}^{I_4-1} x_{i_1 i_2 i_3 i_4} = 1$
for all $i_1$, $i_2$ and $i_3$.
Then the three-facet Rasch model for $\bm{x}$ is expressed by 
\begin{equation}
 \label{model:3-facet}
 P(x_{i_1 i_2 i_3 i_4}=1) = 
 \frac{
 \exp\left[ 
 i_4(\beta_{i_1} - \beta_{i_2} - \beta_{i_3}) - \beta_{i_4}
 \right]
 }
 {
 \sum_{i_4=0}^{I_4-1}
 \exp\left[ 
 i_4(\beta_{i_1} - \beta_{i_2} - \beta_{i_3}) - \beta_{i_4}
 \right]
 }.
\end{equation}
In general, the $V$-facet Rasch model is defined as follows.
Let $\bm{x} = \{x(i)\}$, 
$i := (i_1,\ldots,i_{V+1})$ be an 
$I_1 \times \cdots \times I_{V+1}$ zero-one table. 
Assume that ${\mathcal I}_{V+1} = \{0,\ldots,I_{V+1}-1\}$ and that 
$\bm{x}$ satisfies 
\[
x(i_{\{1, \ldots, V\}}) := \sum_{i_{V+1}=0}^{I_{V+1}-1} x(i)=1.
\]
Then the $V$-facet Rasch model is expressed as
\begin{equation}
 \label{model:many-facet}
 P(x(i)=1) = 
 \frac{
 \exp\left[ 
 i_{V+1}(\beta_{i_1} - \beta_{i_2} - \ldots - \beta_{i_{V}}) - \beta_{i_{V+1}}
 \right]
 }
 {
 \sum_{i_{V+1}=0}^{I_{V+1}-1}
 \exp\left[ 
 i_{V+1}(\beta_{i_1} - \beta_{i_2} - \ldots - \beta_{i_V}) - \beta_{i_{V+1}}
 \right]
 }.
\end{equation}
When $V=2$, $I_3=2$ and $\beta_{i_3} = \mathrm{const}$ 
for $i_3 \in \{0,1\}$, the model coincides with the %single-facet  
Rasch model 
(\ref{model:Rasch_1dim}).
Define $\bm{t}^0$ by 
\[
 \bm{t}^0 = 
  \left\{
   \left.
   \sum_{i_{V+1}=0}^{I_{V+1}-1} i_{V+1} \cdot x(i_{\{v,V+1\}}) \; 
   \right\vert \; 
  i_{\{v,V+1\}} \in {\mathcal I}_{\{v,V+1\}}, \; 
  v=1,\ldots,V
  \right\}. 
\]  
Then the sufficient statistic $\bm{t}$ is written by 
\[
  \bm{t} = \bm{t}^0 \cup 
  \{x(i_{V+1}) \mid i_{V+1} \in {\mathcal I}_{V+1}\}.
\]
When $\beta_{i_{V+1}}$ is constant for 
$i_{V+1} \in {\mathcal I}_{V+1}$, $\bm{t}$ is written by 
\[
  \bm{t} = \bm{t}^0 \cup 
  \{x^+\}, 
\]
where $x^+ := \sum_{i \in {\mathcal I}} x(i)$.
In the case of the three-facet Rasch model (\ref{model:3-facet}), 
$\bm{t}$ is expressed as follows, 
\begin{align*}
 \bm{t} & = 
 \left\{
 \sum_{i_4 = 0}^{I_4} i_4 x_{i_1 + + i_4},\;
 i_1 \in {\mathcal I}_1, \quad
 \sum_{i_4 = 0}^{I_4} i_4 x_{+ i_2 + i_4},\;
 i_2 \in {\mathcal I}_2, 
 \right. \\
 & \qquad\qquad\qquad
 \sum_{i_4 = 0}^{I_4} i_4 x_{+ + i_3 i_4},  \; 
 i_3 \in {\mathcal I}_3, \quad
 x_{+++ i_4}, \; i_4 \in {\mathcal I}_4 
 \Biggr\}.
\end{align*}

In order to implement exact tests for the many-facet Rasch
model, we need a set of moves which connects any fiber 
$\tilde{\mathcal F}_{\bm{t}}$ of zero-one tables.  
In general, however, it is not easy to derive such a set of moves.
As seen in the previous section, in the case of the Rasch model
(\ref{model:Rasch_1dim}),  
the set of basic moves for two-way complete independence model connects
any fiber. 
For the many-facet Rasch model (\ref{model:many-facet}), however, the
basic moves do not necessarily connect all fibers. 
Consider the case where $V=3$ and $I_{4}=2$. 
In this case, $\bm{t}^0$ is written as
\[
  \bm{t}^0 = 
  \{
  x_{i_1 + + 1}, x_{+ i_2 + 1}, x_{+ + i_3 1} \mid 
  i_{v} \in {\mathcal I}_{v}, v=1,2,3
  \}.
\]
Since $x(i_4) = \sum_{i_v \in {\mathcal I}_v} x(i_{\{v,4\}})$ for 
$v=1,2,3$, 
$\bm{t}^0$ is the sufficient statistic. 
$\bm{t}^0$ is equivalent to the sufficient statistics of three-way
complete independence model for $(i_4=1)$-slice of $\bm{x}$.
From Proposition \ref{prop:graver}, the set of square-free moves of the
Graver basis for three-way complete independence
model connects any fiber $\tilde{\mathcal F}_{\bm{t}}$. 
Table \ref{tab:graver} shows the number of square-free moves of the
Graver basis for $I_1 \times I_2 \times I_3$ three-way complete
independence model computed via 4ti2 (\cite{4ti2}).
We see that when the number of levels is greater than two, 
the sets include moves with degree greater than two.
This fact does not necessarily imply that higher degree moves are
required to connect every  fiber for the three-way complete independence
model.  
However we can give an example which shows that the degree two moves do
not connect all fibers of the three-way complete independence model. 

\begin{table}[htbp]
 \centering
 \caption{The number of square-free moves of the Graver basis for
 three-way complete independence model} 
 \label{tab:graver}
 \begin{tabular}{crrrrr} \hline
  & \multicolumn{5}{c}{degree of moves}\\
  $I_1 \times I_2 \times I_3$ & 2 & 3 & 4 & 5 & 6\\ \hline
  $2 \times 2 \times 2$ & 12 & 0 & 0 & 0 & 0\\
  $2 \times 2 \times 3$ & 33 & 48 & 0 & 0 & 0\\
  $2 \times 2 \times 4$ & 64 & 192 & 96 & 0 & 0\\ 
  $2 \times 2 \times 5$ & 105 & 480 & 480 & 0 & 0\\ 
  $2 \times 3 \times 3$ & 90 & 480 & 396 & 0 & 0\\ 
  $2 \times 3 \times 4$ & 174 & 1632 & 5436 & 1152 & 0\\ 
  $2 \times 3 \times 5$ & 285 & 3840 & 23220 & 33120 & 720\\ 
  $3 \times 3 \times 3$ & 243 & 3438 & 19008 & 12312 & 0\\ \hline
 \end{tabular}
\end{table}

\begin{example}[A fiber for $3 \times 3 \times 3$ three-way complete
 independence model] 
Consider the following two zero-one tables $\bm{x}$ and $\bm{y}$ in the
same fiber of three-way complete independence model. 
\[
 \bm{x} := 
 \begin{array}{cc|ccc|} 
  & \multicolumn{1}{c}{~}  &   & k &  \multicolumn{1}{c}{~} \\
  & \multicolumn{1}{c}{~} & 1 & 2 &  \multicolumn{1}{c}{3}\\ \cline{3-5}
  & 1 & 0 & 0 & 0\\
j & 2 & 0 & 0 & 1\\
  & 3 & 0 & 0 & 1\\ \cline{3-5}
  & \multicolumn{1}{c}{~}  & \multicolumn{3}{c}{i=1}
 \end{array} \quad 
 \begin{array}{cc|ccc|} 
  & \multicolumn{1}{c}{~}  &   &  &  \multicolumn{1}{c}{~} \\
  & \multicolumn{1}{c}{~} &  &  &  \multicolumn{1}{c}{~}\\ \cline{3-5}
  &  & 0 & 1 & 1\\
  &  & 0 & 1 & 1\\
  &  & 1 & 1 & 1\\ \cline{3-5}
  & \multicolumn{1}{c}{~}  & \multicolumn{3}{c}{i=2}
 \end{array} \quad 
 \begin{array}{cc|ccc|} 
  & \multicolumn{1}{c}{~}  &   &  &  \multicolumn{1}{c}{~} \\
  & \multicolumn{1}{c}{~} &  &  &  \multicolumn{1}{c}{}\\ \cline{3-5}
  &  & 0 & 0 & 0\\
 &  & 0 & 0 & 1\\
  &  & 1 & 1 & 1\\ \cline{3-5}
  & \multicolumn{1}{c}{~}  & \multicolumn{3}{c}{i=3}
 \end{array} \quad 
\]
\[
 \bm{y} := 
 \begin{array}{cc|ccc|} 
  & \multicolumn{1}{c}{~}  &   & k &  \multicolumn{1}{c}{~} \\
  & \multicolumn{1}{c}{~} & 1 & 2 &  \multicolumn{1}{c}{3}\\ \cline{3-5}
  & 1 & 0 & 0 & 0\\
j & 2 & 0 & 0 & 0\\
  & 3 & 0 & 1 & 1\\ \cline{3-5}
  & \multicolumn{1}{c}{~}  & \multicolumn{3}{c}{i=1}
 \end{array} \quad 
 \begin{array}{cc|ccc|} 
  & \multicolumn{1}{c}{~}  &   &  &  \multicolumn{1}{c}{~} \\
  & \multicolumn{1}{c}{~} &  &  &  \multicolumn{1}{c}{~}\\ \cline{3-5}
  &  & 0 & 0 & 1\\
  &  & 1 & 1 & 1\\
  &  & 1 & 1 & 1\\ \cline{3-5}
  & \multicolumn{1}{c}{~}  & \multicolumn{3}{c}{i=2}
 \end{array} \quad 
 \begin{array}{cc|ccc|} 
  & \multicolumn{1}{c}{~}  &   &  &  \multicolumn{1}{c}{~} \\
  & \multicolumn{1}{c}{~} &  &  &  \multicolumn{1}{c}{}\\ \cline{3-5}
  &  & 0 & 0 & 1\\
 &  & 0 & 0 & 1\\
  &  & 0 & 1 & 1\\ \cline{3-5}
  & \multicolumn{1}{c}{~}  & \multicolumn{3}{c}{i=3}
 \end{array}.
\]
The difference of the two tables is 
\[
 \bm{z} =
 \begin{array}{|ccc|}\hline
  0 & 0 & 0\\
  0 & 0 & 1\\
  0 & -1 & 0\\ \hline
 \end{array} \quad 
 \begin{array}{|ccc|}\hline 
  0 & 1 & 0\\
  -1 & 0 & 0\\
  0 & 0 & 0\\ \hline
 \end{array} \quad 
 \begin{array}{|ccc|}\hline 
  0 & 0 & -1\\
  0 & 0 & 0\\
  1 & 0 & 0\\ \hline
 \end{array}
\]
and we can easily check that $\bm{z}$ is a move for the three-way
complete independence model.

Let $\bar\Delta$ be the set of degenerate variables defined in
\cite{decomposable}. 
Then degree two moves for three-way complete independence model are
 classified into the following four patterns.
\begin{enumerate}
 \item $\bar\Delta = \{1,2,3\}$ $:$
      $ \displaystyle{
      \begin{array}{r|rr|} 
       \multicolumn{1}{r}{~}& i_3 & \multicolumn{1}{r}{i_3^\prime} \\ \cline{2-3}
       i_2 & 1 & 0\\
       i_2^\prime & 0 & -1\\ \cline{2-3}
       \multicolumn{1}{r}{~} & \multicolumn{2}{r}{i_1}
      \end{array}, \quad 
      \begin{array}{r|rr|} 
       \multicolumn{1}{r}{~}& i_3 & \multicolumn{1}{r}{i_3^\prime} \\ \cline{2-3}
       i_2 & -1 & 0\\
       i_2^\prime & 0 & 1\\ \cline{2-3}
       \multicolumn{1}{r}{~} & \multicolumn{2}{r}{i_1^\prime}
      \end{array}}\; ; 
      $
 \item $\bar\Delta = \{1,2\}$ $:$
      $ \displaystyle{
      \begin{array}{r|rr|} 
       \multicolumn{1}{r}{~}& i_3 & \multicolumn{1}{r}{i_3^\prime} \\ \cline{2-3}
       i_2 & 1 & 0\\
       i_2^\prime & -1 & 0\\ \cline{2-3}
       \multicolumn{1}{r}{~} & \multicolumn{2}{r}{i_1}
      \end{array}, \quad 
      \begin{array}{r|rr|} 
       \multicolumn{1}{r}{~}& i_3 & \multicolumn{1}{r}{i_3^\prime} \\ \cline{2-3}
       i_2 & 0 & -1\\
       i_2^\prime & 0 & 1\\ \cline{2-3}
       \multicolumn{1}{r}{~} & \multicolumn{2}{r}{i_1^\prime}
      \end{array}}\; ;
       $
 \item $\bar\Delta = \{1,3\}$ $:$
      $ \displaystyle{
      \begin{array}{r|rr|} 
       \multicolumn{1}{r}{~}& i_3 & \multicolumn{1}{r}{i_3^\prime} \\ \cline{2-3}
       i_2 & 1 & -1\\
       i_2^\prime & 0 & 0\\ \cline{2-3}
       \multicolumn{1}{r}{~} & \multicolumn{2}{r}{i_1}
      \end{array}, \quad 
      \begin{array}{r|rr|} 
       \multicolumn{1}{r}{~}& i_3 & \multicolumn{1}{r}{i_3^\prime} \\ \cline{2-3}
       i_2 & -1 & 1\\
       i_2^\prime & 0 & 0\\ \cline{2-3}
       \multicolumn{1}{r}{~} & \multicolumn{2}{r}{i_1^\prime}
      \end{array}}\; ;
       $
 \item $\bar\Delta = \{2,3\}$ $:$
      $ \displaystyle{
       \begin{array}{r|rr|} 
       \multicolumn{1}{r}{~}& i_3 & \multicolumn{1}{r}{i_3^\prime} \\ \cline{2-3}
       i_2 & 1 & -1\\
       i_2^\prime & -1 & 1\\ \cline{2-3}
       \multicolumn{1}{r}{~} & \multicolumn{2}{r}{i_1}
      \end{array}}.
       $
\end{enumerate}
However it is easy to check that if we apply any move in this class to 
$\bm{x}$ or $\bm{y}$, $-1$ or $2$ has to appear.
Therefore we cannot apply any degree two moves to both $\bm{x}$
and $\bm{y}$. 
Hence a degree three move is required to connect this fiber.
\end{example}

This example indicates that it may be difficult to obtain a set of
moves which connects every fiber of the many-facet Rasch model
theoretically.   
As seen in Table \ref{tab:graver}, the number of square-free moves in
the Graver basis is too large even for three-way tables.
When the number of cells is greater than 100, it seems to be difficult
to compute the Graver basis via 4ti2 in a practical length of time.  
Hence implementations of exact tests by using the Graver basis is limited
to very small models at this point.  
The clarification of the structure of the set of moves which connects
all zero-one fibers for more general many-facet Rasch model 
is important to implement exact tests. 
However this problem seems to be difficult at this point and is left as
a future task.   

\subsection{Two-way zero-one tables with structural zeros}
Two-way zero-one tables with structural zeros arise in many practical
problems, including ecological studies and social networks.
Let $\bm{x}=\{x_{ij}\}$ be an $I \times J$ zero-one table and denote by 
$S \subset {\mathcal I}$  the set of cells that are not structural
zeros. 
We consider the quasi-independence model  (\cite{BFH1975}) as a null
hypothesis,  
\begin{equation}
 \label{model:quasi-indep}
 \left\{
 \begin{array}{rll}
 \log P(x_{ij}=1) & =  
   \mu + \alpha_{i} + \beta_{j}, &
   (i,j) \in S\\
 P(x_{ij}=1) & = 0, & \text{otherwise}.
 \end{array}
 \right.
\end{equation}
The sufficient statistic $\Bt$ for the models is the set of row and
column sums. 
Denote by ${\mathcal B}(S)$ the set of moves for the quasi-independence model 
(\ref{model:quasi-indep}). 
Then ${\mathcal B}(S)$ is written by 
\[
 {\mathcal B}(S) = \{\bm{z} = \{z_{ij}\} \mid z_{i+} = z_{+j} = 0, \; 
 z_{ij}=0 \; \text{ for } \; (i,j) \in S^c\}.
\]
We denote a structural zero cell by $[0]$ to distinguish it from a
sampling zero cell.     

\cite{rao-etal-1996sankhya} discussed the connectivity of zero-one
tables in the case where $I=J$ and all the diagonal elements are
structural zeros and provided a Markov basis for zero-one tables.
\cite{roberts-2000SocialNetworks} applied the results to analyses of
social networks and proposed an efficient implementation of exact tests
of quasi-independent model (\ref{model:quasi-indep}) via MCMC. 
\cite{Chen-2007JCGS} proposed a procedure for implementing exact tests
via sequential importance sampling for general two-way zero-one tables
with structural zeros. 
In this section we extend the argument of \cite{rao-etal-1996sankhya} 
to general two-way zero-one tables with structural zeros and provide a
Markov basis for zero-one tables in the quasi-independence model.

For general two-way contingency tables, \cite{aoki-takemura-2005jscs}
provided a complete description of the unique minimal Markov basis for
the quasi-independence model (\ref{model:quasi-indep}).  
A loop $\bm{z}_r(i_{[r]}; j_{[r]})$ is defined as in Section
\ref{subsec:Rasch}.  
When $\bm{z}_r(i_{[r]}; j_{[r]})$ is a move in ${\mathcal B}(S)$, 
$\bm{z}_r(i_{[r]}; j_{[r]})$ is called a loop on $S$. 

\begin{definition}[\cite{aoki-takemura-2005jscs}]
 A loop $\bm{z}_r(i_{[r]}; j_{[r]})$ on $S$ is called df 1 if 
 ${\mathcal I}(i_{[r]}; j_{[r]})$ 
 does not contain  support of any loop on $S$ of 
 degree $2,\ldots ,r-1$, where
 \[
 {\mathcal I}(i_{[r]}; j_{[r]}) =
 \{(i,j)\ |\ i \in \{i_1,\ldots ,i_r\},j\in\{j_1,\ldots ,j_r\}\}.
 \]
 $\bm{z}_r(i_{[r]}; j_{[r]})$ is df 1 if and only if
 ${\mathcal I}(i_{[r]}; j_{[r]})$ contains exactly two elements
 in $S$ in every row and column.
\end{definition}

The following integer arrays are examples of 
df 1 loops of degree two, three and four on some $S$.
\begin{equation}
 \label{df1-loop}
\begin{array}{|ccccc|}\hline
  +1           & -1  & 0 & 0 & 0\\
  -1           & +1  & 0 & 0 & 0\\
  0            & 0   & 0 & 0 & 0\\
  0            & 0   & 0 & 0 & 0\\ \hline
\end{array}
\hspace*{5mm}
\begin{array}{|ccccc|}\hline
  +1           & -1  & [0]& 0 & 0\\
  -1           & [0] & +1 & 0 & 0\\
  \mbox{$[0]$} & +1  & -1 & 0 & 0\\
  0            & 0   &  0 & 0 & 0\\ \hline
\end{array}
\hspace*{5mm}
\begin{array}{|ccccc|}\hline
  +1           & -1  & [0]& [0]& 0\\
  -1           & [0] & +1 & [0]& 0\\
  \mbox{$[0]$} & +1  & [0]& -1 & 0\\
  \mbox{$[0]$} & [0] & -1 & +1 & 0\\ \hline
\end{array}
\end{equation}
We note that a degree $2$ loop $\bm{z}_2(i_1,i_2; j_1,j_2)$ is a basic
move. 

Denote by ${\mathcal B}_{\mathrm{df}1}(S)$ the set of 
df 1 loops of degree $2,\ldots ,\min\{I,J\}$. 
For general contingency tables, \cite{aoki-takemura-2005jscs} showed
that ${\mathcal B}_{\mathrm{df}1}(S)$ forms two-way unique minimal
Markov basis for the quasi-independence model (\ref{model:quasi-indep}).  
By following the argument in \cite{aoki-takemura-2005jscs}, however, 
we can also prove that ${\mathcal B}_{\mathrm{df}1}(S)$ connects 
every fiber $\tilde{\mathcal F}_\Bt$ of zero-one tables. 

\begin{theorem}
 \label{thm:main-str-zeros}
 The set of df $1$ loops of degree $2,\ldots ,\min\{I,J\}$ connects
 every fiber $\tilde{\mathcal F}_\Bt$ of zero-one tables of 
 the quasi-independence model (\ref{model:quasi-indep}). 
\end{theorem}

The proof is in the same way as the proof of Theorem 1 in 
\cite{aoki-takemura-2005jscs} and omitted here.

As discussed in Section 5 in \cite{aoki-takemura-2005jscs}, 
in the case of square tables with diagonal elements being structural
zeros, 
${\mathcal B}_{\mathrm{df}1}(S)$ contains basic moves and df 1 loops
of degree 3 which coincides with the results of
\cite{rao-etal-1996sankhya}.

\subsection{Latin squares and zero-one tables for no-three-factor-interaction models}
\label{sec:latin-squares}

Zero-one tables also appear quite often in the form of incidence matrices
for combinatorial problems.  Here as an example 
% we consider
% Latin squares as tables of a particular fiber of zero-one tables with
% all two-dimensional mariginals (line sums) equal to one.
we consider Latin squares.
A Latin square is an $n \times n$ table filled with $n$ different symbols in
such a way that each symbol occurs exactly once in each row and column. 
A $3 \times 3$ Latin square is written by 
\begin{equation}
 \label{3*3-LS}
\begin{array}{|ccc|}\hline
 1 & 2 & 3\\
 2 & 3 & 1\\
 3 & 1 & 2\\ \hline
\end{array} .
\end{equation}
When the symbols of an $n\times n$  Latin square are considered as
coordinates of the third axis (sometimes called the orthogonal array
representation of a Latin square), it is a particular element of a fiber
for the $n\times n\times n$ no-three-factor-interaction model with all
two-dimensional marginals (line sums) equal to 1.
%(\cite{jacobson-matthews}).  
For example, the $3 \times 3$ Latin square (\ref{3*3-LS}) is considered as 
a $3 \times 3 \times 3$ zero-one table $\bm{x}=\{x_{i_1 i_2 i_3}\}$
\begin{equation}
 \label{3*3-LT-2}
\bm{x} =
\begin{array}{ccc}
\begin{array}{c|ccc|}
 \multicolumn{1}{c}{} &\multicolumn{1}{c}{} & i_2 & \multicolumn{1}{c}{} 
  \\  \cline{2-4}
& 1 & 0 & 0\\
i_1 & 0 & 0 & 1\\
& 0 & 1 & 0\\ \cline{2-4}
\end{array}, & 
\begin{array}{c|ccc|}
 \multicolumn{1}{c}{} &\multicolumn{1}{c}{} & i_2 & \multicolumn{1}{c}{} 
  \\  \cline{2-4}
& 0 & 1 & 0\\
i_1 & 1 & 0 & 0\\
& 0 & 0 & 1\\ \cline{2-4}
\end{array}, & 
\begin{array}{c|ccc|}
 \multicolumn{1}{c}{} &\multicolumn{1}{c}{} & i_2 & \multicolumn{1}{c}{} 
  \\  \cline{2-4}
& 0 & 0 & 1\\
i_1 & 0 & 1 & 0\\
& 1 & 0 & 0\\ \cline{2-4}
\end{array}\\
\qquad i_3=1 & \qquad i_3=2 & \qquad i_3=3 
\end{array}
\end{equation}
with $x_{i_1 i_2 +} = 1$, $x_{+ i_2 i_3} = 1$, $x_{i_1 + i_3} = 1$ 
for all $i_1$, $i_2$ and $i_3$.
One of the reasons to consider a Markov basis for Latin squares is to
generate a Latin square randomly.  \cite{fisher-yates} advocated to
choose a Latin square randomly from the set of Latin squares.  
\cite{jacobson-matthews} gave a Markov basis for 
the set of $n\times n$ Latin squares.

Because the set of Latin squares is just a particular fiber, it may be
the case that a minimal set of moves connecting all Latin squares is
smaller to the set of moves connecting all zero-one tables.  This is
indeed the case as we show for the simple case of $n=3$.
We first present a connectivity result for $3\times 3\times 3$ zero-one tables with
all line sums fixed. 
%We use similar notation to the last subsection.

Let $\Bz=\{z_{ijk}\}_{i,j,k=1,2,3}$ be a move for  $3\times 3\times 3$ 
no-thee-factor-interaction model.  From \cite{diaconis-sturmfels} and
\cite{aoki-takemura-2003anz}
the minimal Markov basis consists of basic moves such as
\begin{equation}
\label{eq:333basic}
 \bm{z} =
 \begin{array}{|ccc|}\hline
  1 & -1 & 0\\
  -1 & 1 & 0\\
  0 & 0 & 0\\ \hline
 \end{array} \quad 
 \begin{array}{|ccc|}\hline 
  -1 & 1 & 0\\
  1 & -1 & 0\\
  0 & 0 & 0\\ \hline
 \end{array} \quad 
 \begin{array}{|ccc|}\hline 
  0 & 0 & 0\\
  0 & 0 & 0\\
  0 & 0 & 0\\ \hline
 \end{array}
\end{equation}
and degree 6 moves such as
\begin{equation}
\label{eq:333deg6}
 \bm{z} =
 \begin{array}{|ccc|}\hline
  1 & -1 & 0\\
  0 & 1 & -1\\
  -1& 0 & 1\\ \hline
 \end{array} \quad 
 \begin{array}{|ccc|}\hline 
  -1 & 1 & 0\\
  0 & -1 & 1\\
  1 & 0 & -1\\ \hline
 \end{array} \quad 
 \begin{array}{|ccc|}\hline 
  0 & 0 & 0\\
  0 & 0 & 0\\
  0 & 0 & 0\\ \hline
 \end{array} \ .
\end{equation}
However these moves do not connect zero-one tables of the
$3\times 3\times3$ no-three-factor-interaction model.    We need the
following type of degree 9 move, which corresponds to the difference of two
Latin squares.
\begin{equation}
\label{eq:333deg9}
 \bm{z} =
 \begin{array}{|ccc|}\hline
  1 & -1 & 0\\
  0 & 1 & -1\\
  -1& 0 & 1\\ \hline
 \end{array} \quad 
 \begin{array}{|ccc|}\hline 
  0 & 1 & -1\\
  -1 & 0 & 1\\
  1 & -1 & 0\\ \hline
 \end{array} \quad 
 \begin{array}{|ccc|}\hline 
  -1 & 0 & 1\\
  1 & -1 & 0\\
  0 & 1 & -1\\ \hline
 \end{array} \ .
\end{equation}

\begin{proposition}
\label{prop:333}
The set of basic moves (\ref{eq:333basic}), degree 6 moves
(\ref{eq:333deg6}) and degree 9 moves (\ref{eq:333deg9}) forms a Markov
basis for $3\times 3\times 3$ zero-one tables for the no-three-factor-interaction
model.
\end{proposition}

\begin{proof}
Consider any line sum, such as $0=z_{+11}=z_{111}+z_{211}+z_{311}$ 
of a move $\Bz$.  If $(z_{111},z_{211},z_{311}) \neq (0,0,0)$, then
we easily see that $\{z_{111},z_{211},z_{311}\} = \{ -1,0,1\}$.
By a similar consideration as in  \cite{aoki-takemura-2003anz},
each $i$- or $j$- or $k$-slice is either a loop of degree two or
loop of degree three, such as
\begin{equation}
\label{eq:333slice}
 \begin{array}{|ccc|}\hline
  1 & -1 & 0\\
  -1 & 1 & 0\\
  0 & 0 & 0\\ \hline
 \end{array} \quad  \text{or} \quad
 \begin{array}{|ccc|}\hline
  1 & -1 & 0\\
  0 & 1 & -1\\
  -1 & 0 & 1\\ \hline
\end{array}\  .
\end{equation}
Now we consider two cases: 1) there exists a slice with a loop of
degree two, or 2) all slices are loops of degree three.\\
{\bf Case 1.} 
Without loss of generality, we can assume that the $(i=1)$-slice of
$\Bz$ is the loop of degree two in (\ref{eq:333slice}).  Then we can
further assume that $z_{211}=-1$ and $z_{311}=0$.  Now suppose
that $z_{222}=-1$. 
If $z_{212}=1$ or $z_{221}=1$, then
this constitutes a strong crossing pattern of Condition \ref{cond:1}
and we can reduce $|\Bz|$ by a basic move.  This implies 
$z_{212}=z_{221}=0$. But then $z_{213}=z_{223}=1$ and
this contradicts the pattern of $\{z_{213},z_{223},
z_{233}\}=\{-1,0,1\}$. 

By the above consideration we have $z_{222}=0$ and $z_{322}=-1$.
By a similar  consideration for the cells $z_{i12}$ and $z_{i21}$,
$i=1,2,3$, we easily see that $\Bz$ is of the form
\[
 \begin{array}{|ccc|}\hline
  1 & -1 & 0\\
  -1 & 1 & 0\\
  0 & 0 & 0\\ \hline
 \end{array} \quad  
 \begin{array}{|ccc|}\hline
  -1 & 1 & 0\\
  0 & 0 & 0\\
  1 & -1 & 0\\ \hline
\end{array} \quad
\begin{array}{|ccc|}\hline
  0 & 0 & 0\\
  1 & -1 & 0\\
  -1 & 1 & 0\\ \hline
 \end{array}, 
\]
which is a degree 6 move.\\
{\bf Case 2.}   It is easily seen that the only case where degree 6
moves can not be applied is of the form of the move of degree 9 in
(\ref{eq:333deg9}).  This proves that connectivity is guaranteed if we
add degree 9 moves.  

We also want to show that degree 9 moves are  needed 
for connectivity.  Consider
\[
\Bx = 
 \begin{array}{|ccc|}\hline
  1 & 0 & 1\\
  0 & 1 & 0\\
  0& 0 & 1\\ \hline
 \end{array} \quad 
 \begin{array}{|ccc|}\hline 
  0 & 1 & 0\\
  0 & 1 & 1\\
  1 & 0& 0\\ \hline
 \end{array} \quad 
 \begin{array}{|ccc|}\hline 
  0 & 0 & 1\\
  1 & 0 & 0\\
  1 & 1 & 0\\ \hline
 \end{array} \ .
\]
By a simple program it is easily checked that if we apply any basic move
or any move of degree 6 to $\Bx$, $-1$ or $2$ has to appear. Hence degree
9 moves are required to connect zero-one tables.

\end{proof}

Now consider $3\times 3$ Latin squares (\ref{3*3-LT-2}).  It is
well-known that there is only one isotopy class of $3\times 3$ Latin
squares (Chapter III of \cite{colbourn-dinitz-2nd}), i.e., all 
$3 \times 3$ Latin squares are connected by the action of the direct
product $S_3 \times S_3 \times S_3$ of the symmetric group $S_3$ %,
%which corresponds to permutation of 
%levels of a factor of three-way contingency tables.  Note that the symmetric
%group 
which 
is generated by transpositions, and a transposition corresponds
to a move of degree 6 in (\ref{eq:333deg6}). Therefore, 
%we have the following fact: 
{\em the set of $3\times 3$ Latin squares in the
orthogonal array representation is connected by the set of 
moves of degree 6 in (\ref{eq:333deg6})}.   In view of Proposition
\ref{prop:333}, we see that we do not need basic moves nor degree 9 moves
for connecting $3\times 3$ Latin squares.

There are two isotopy classes for $4\times 4$ Latin squares (1.18 of
III.1.3 of \cite{colbourn-dinitz-2nd}) and representative elements of
these two classes are connected by a basic move.  Transposition of
two levels for a factor corresponds to a degree 8 move of the following form.
{\small
\[
 \bm{z} =
 \begin{array}{|cccc|}\hline
  1 & -1 & 0 & 0\\
  0 & 1 & -1 & 0\\
  0& 0 & 1 & -1\\ 
  -1 & 0 & 0 & 1\\ \hline
 \end{array} \quad 
 \begin{array}{|cccc|}\hline
  -1 & 1 & 0 & 0\\
  0 & -1 & 1 & 0\\
  0& 0 & -1 & 1\\ 
  1 & 0 & 0 & -1\\ \hline
 \end{array} \quad 
 \begin{array}{|cccc|}\hline 
  0 & 0 & 0 & 0\\
  0 & 0 & 0 & 0\\
  0 & 0 & 0 & 0\\
  0 & 0 & 0 & 0\\ \hline
 \end{array} 
\quad 
\begin{array}{|cccc|}\hline 
  0 & 0 & 0 & 0\\
  0 & 0 & 0 & 0\\
  0 & 0 & 0 & 0\\
  0 & 0 & 0 & 0\\ \hline
 \end{array} \ .
\]
}
Therefore the set of $4\times 4$ Latin squares is connected by
the set of basic moves and moves of degree 8 of the above form.
We can apply a similar consideration to the celebrated result of 
22 isotopy classes of $6\times 6$ Latin squares 
derived by \cite{fisher-yates}.

\section{Concluding remarks}
\label{sec:remarks}

In this paper we discussed Markov bases for tables with zero-one
entries.  We derived several general results, where a particular subset
of the Graver basis connects zero-one tables.  However, in general, we
found that a Markov basis for zero-one tables is difficult and requires 
separate arguments for each model.  We obtained 
Markov bases for zero-one tables for some common models of
contingency tables.

\cite{rapallo-yoshida} gave some results for contingency tables with
bounded entries, in particular for the case of 
two-way tables with structural zeros.
A zero-one table is a particular case of contingency
tables with bounded entries.    If the bound is large
enough, compared to the sample size of a particular fiber, it seems
that the bound is not binding.
In this sense the bound of 1 in our case seems to be most stringent.
On the other hand, our proof of Proposition \ref{prop:333} suggests that
a Markov basis for zero-one tables may have a simple structure. 
It is an interesting problem to describe how Markov bases behave 
as we vary the upper bound for the cells.

In Section \ref{sec:latin-squares} we considered Latin squares.  It is
of interest to consider other combinatorial designs, such as the Sudoku.
A Markov basis for Sudoku designs is considered in
\cite{fontana-rogantin-indicator}
and their invariance structure is discussed in \cite{perturbation}.
%discuss the invariance structure of sudoku.  
It is a challenging problem to derive a Markov basis for the ordinary
$3\times 3\times 3\times 3$ Sudoku.

\section*{Acknowledgment}
The authors would like to thank the editors and a anonymous referee for
giving us constructive comments to improve this paper.

\bibliographystyle{plainnat}
\bibliography{Hara-Takemura-arxiv0908.4461}

\end{document}